\documentclass[12pt]{amsart}

\usepackage{amssymb,amsfonts,amsmath,amsopn,amstext,amscd,latexsym,fullpage,verbatim,graphicx}
\usepackage{amsthm}

\newtheorem{Proposition}{Proposition}
\newtheorem{Theorem}[Proposition]{Theorem}
\newtheorem{Corollary}[Proposition]{Corollary}
\newtheorem{Lemma}[Proposition]{Lemma}

\newtheorem*{thm*}{Theorem}

\numberwithin{equation}{section}

\DeclareMathOperator{\GL}{GL}
\DeclareMathOperator{\gr}{gr}

\DeclareMathOperator{\Hom}{Hom}

\DeclareMathOperator{\Ind}{Ind}

\DeclareMathOperator{\Lie}{Lie}

\DeclareMathOperator{\Rep}{Rep}

\newcommand{\bbP}{{\mathbb P}}

\newcommand{\bbZ}{{\mathbb Z}}

\newcommand{\frd}{{\mathfrak d}}

\newcommand{\frg}{{\mathfrak g}}

\newcommand{\frp}{{\mathfrak p}}

\newcommand{\cE}{{\mathcal E}}
\newcommand{\cF}{{\mathcal F}}

\newcommand{\cL}{{\mathcal L}}

\newcommand{\cO}{{\mathcal O}}

\newcommand{\cX}{{\mathcal X}}
\newcommand{\cY}{{\mathcal Y}}

\newcommand{\Qp}{\mathbb Q_p} % p-adische Zahlen
\newcommand{\Z}{\mathbb Z}

\renewcommand{\P}{\mathbb P}

\DeclareMathOperator{\alg}{alg}

\newcommand{\Oa}{\cO_{\alg}}

\begin{document}

\title{The de Rham cohomology of Drinfeld's half space}
\author{Sascha Orlik}
\address{Fachgruppe Mathematik, Bergische Universit\"at Wuppertal, Gau{\ss}{}stra\ss{}e 20,
D-42119 Wuppertal,  Germany}
\email{orlik@math.uni-wuppertal.de}
\date{}

\begin{abstract}
Let $\cX \subset \P_K^d$ be Drinfeld's  half space over a $p$-adic field $K$. The de Rham cohomology of $\cX$ was first
computed by Schneider and Stuhler \cite{SS}. Afterwards there were given different proofs by Alon, de Shalit, Iovita and 
Spiess \cite{AdS,dS,IS}. This paper presents yet another approach for the determination of these invariants by analysing the de 
Rham complex of $\cX$ from the viewpoint of results given  in \cite{O}, \cite{OS}. Moreover, we treat as a generalization
the dual BGG complex of a given algebraic representation in the sense of Faltings \cite{Fa} respectively 
Schneider \cite{S}. 
\end{abstract} 

\maketitle

\normalsize

\section{Introduction}
Let $p$ be a prime number and let $K$ be a finite extension of the field of $p$-adic numbers $\Qp$.
We denote by $\cX=\cX_K^{(d+1)}=\mathbb \P_K^d\setminus \bigcup\nolimits_{H \varsubsetneq K^{d+1}}\mathbb \P(H)$ 
(the complement of all  $K$-rational hyperplanes in projective space)
Drinfeld's half space  \cite{D} of dimension $d\geq1$ over $K$. It is a rigid analytic variety over $K$ which is equipped with an action of
the $p$-adic Lie group  $G={\rm GL}_{d+1}(K)$.
In \cite{SS} Schneider and Stuhler determined the cohomology of $\cX$ for any ``good'' cohomology theory 
(e.g. the \'etale and the de Rham cohomology) as $G$-representations. Here they make only use of the ``good'' properties as homotopy invariance, existence of a product structure 
etc.  It turns out that the de Rham cohomology is given by
\begin{equation}\label{cohomology}
 H^\ast_{\rm dR}(\cX)=\bigoplus_{i=0}^d \Hom_K(v^G_{P_{(d+1-i,1,\ldots,1)}},K)[-i].
\end{equation}
Here $P_{(d+1-i,1,\ldots,1)}$ is the (lower) standard parabolic subgroup of $G$ which corresponds to the decomposition 
$(d+1-i,1,\ldots,1)$ of $d+1.$ Further for a parabolic subgroup $P\subset G$, the smooth generalized Steinberg representation 
$v^G_P$ is the unique irreducible quotient of the smooth unnormalized  induced representation $i^G_P={\rm ind}^G_P(K)$ with 
respect to the trivial $P$-representation \cite{BW,Ca}. A few years later there were given different proofs of this result
by Alon, de Shalit, Iovita and Spiess \cite{AdS,dS,IS} by relating
differential forms on $\cX$ with harmonic cochains on the Bruhat-Tits building of $G$ and considering logarithmic forms, respectively. 
%Finally in \cite{O1,O2} there was more generally computed the \'etale cohomology for (some) $p$-adic period domains. 
%Hence via comparison with de Rham cohomology these latter results lead to the formula (\ref{cohomology}), as well.

In this short notice we explain how we can determine the de Rham cohomology of $\cX$ from its de Rham complex
\begin{equation}\label{deRham}
 \Omega^\bullet(\cX):\, 0 \to \cO(\cX) \to \Omega^1(\cX) \to \cdots \to \Omega^d(\cX) \to 0
\end{equation}
by applying some recent results given in \cite{O,OS}.
Here for $i=0,\ldots,d$, the expression $\Omega^i(\cX)=H^0(\cX, \Omega^i)$ is the space  of $\cX$-valued sections of the usual homogeneous vector bundle $\Omega^i$ on projective space $\P^d_K.$  
Further the de Rham cohomology of $\cX$ is the ordinary cohomology of the above complex since $\cX$ is a Stein space.
In contrast to the generalized Steinberg representations $v^G_P$ the contributions $\Omega^i(\cX)$ in the de Rham complex   are 
much bigger objects. Indeed  they are reflexive $K$-Fr\'echet spaces with a continuous $G$-action \cite{ST1}. 
Their strong duals $\Omega^i(\cX)'$, $i=0,\ldots,d,$ (i.e., the $K$-vector space of continuous linear forms
equipped with the strong topology of bounded convergence) are locally analytic $G$-representations in the sense of Schneider and Teitelbaum \cite{ST2}.
More generally, the same holds true for arbitrary homogeneous vector bundles on $\P_K.$
In \cite{O} there is constructed  for any such homogeneous vector bundle $\cE$,  a decreasing filtration by closed $G$-stable subspaces
\begin{equation}\label{filtration}
 \cE(\cX)^0 \supset \cE(\cX)^{1} \supset \cdots \supset \cE(\cX)^{d-1} \supset \cE(\cX)^d =  H^0(\P^d,\cE)
\end{equation}
on $\cE(\cX)^0=\cE(\cX).$ As we will see in the next section the filtration behaves functorially in $\cE.$ Hence we get a filtered de Rham complex
\begin{equation}\label{filtered_deRham}
 \big(0 \to \cO(\cX)^j \to \Omega^1(\cX)^j  \cdots \to  \Omega^d(\cX)^j  \to  0\big)_{j=0,\ldots,d}\, .
\end{equation}
In this paper we analyse its induced spectral sequence $$E^{p,q}_0=\gr^p(\Omega^{p+q}(\cX)) \Rightarrow H^{p+q}(\Omega^\bullet(\cX)),$$
cf. \cite{EGAIII}. The main theorem of this paper is the following result.
\begin{Theorem}
 The spectral sequence $E_0$ attached to the filtered de Rham complex (\ref{filtered_deRham}) degenerates at $E_1$ and 
 yields  the cohomology 
 formula (\ref{cohomology}).
\end{Theorem}

In the final section we replace the de Rham complex by the dual BGG complex attached to an algebraic representation in the 
sense of Faltings \cite{Fa, FC} respectively 
Schneider \cite{S}. More precisely, let $\lambda\in \bbZ^{d+1}$ be a dominant weight with corresponding
irreducible algebraic representation $V(\lambda).$ Then we consider the complex
\begin{equation}\nonumber
0 \rightarrow V(\lambda) \rightarrow \cE_\lambda (\cX) \rightarrow \cE_{w_1\cdot \lambda}(\cX) \rightarrow \cdots \rightarrow
\cE_{w_d\cdot \lambda}(\cX) \rightarrow 0
\end{equation}
where the $\cE_{w_i\cdot \lambda}$ are certain homogeneous vector bundles on $\bbP^d_K$ depending
on the weight $w_i\cdot \lambda$ (For a precise description we refer to  the final section). 
It is proved in \cite{S} that it is quasi-isomorphic to the complex $\Omega^\bullet(\cX) \otimes V(\lambda).$
It coincides with the de Rham complex (\ref{deRham}) for $\lambda=0.$ In particular the determination of the cohomology of
$\cE_{\bullet\cdot \lambda}(\cX)$ is not a surprising issue. Nevertheless, we get with
the same proof:

\begin{thm*}{\rm \bf 1'.}
Let $\lambda \in X^+.$ Then  the spectral sequence $E_0$ attached to the attached filtered complex degenerates at $E_1$ 
and we get
\begin{equation}\nonumber
 H^\ast(\cE_{\bullet\cdot \lambda}(\cX))=\bigoplus_{i=0}^d \Hom_K(v^G_{P_{(d+1-i,1,\ldots,1)}},V(\lambda))[-i].
\end{equation}
\end{thm*}

\section{The proof of Theorem 1}

\begin{comment}

\vspace{0.7cm}
{\it Notation:} 

We use bold letters $\bf G, \bf P,\ldots $ to denote algebraic group schemes over $K$, whereas we use normal letters
$G,P, \ldots $ for their $K$-valued points of $p$-adic groups.
We use Gothic letters $\frg,\frp,\ldots $ for their Lie algebras. The corresponding enveloping algebras are denoted as usual by $U(\frg), U(\frp), \ldots
.$ Finally, we set ${\bf G}:= {\rm {\bf GL_{d+1}}}.$ 
\end{comment}

%Further, I would like to thank  Benjamin Schraen for indicating a mistake in a previous version.

We begin by recalling some terminology used in \cite{O}. The following lines are an extract of \cite[Section 1]{O}. 

We consider the action of $G$ on projection space $\P^d_K$ given by
$$g\cdot [q_0:\cdots :q_d]:=[q_0:\cdots :q_d]g^{-1} \,.$$
We fix  a  homogeneous vector bundle $\cE$ on
$\P^d_K$ and let $\frg=\Lie G$ be the Lie algebra of $G.$ Then $\cE$ is  naturally a $\frg$-module, i.e., there is a
homomorphism of Lie algebras $\frg \rightarrow {\rm End}(\cE)$  which extends to the universal enveloping algebra $U(\frg).$
Fix an integer $0 \leq j \leq d-1$ and let $$\P^j_K=V(X_{j+1},\ldots,X_{d})\subset \P^d_K$$ be the closed $K$-subvariety
defined by the vanishing of the coordinates $X_{j+1},\ldots,X_{d}.$ 
Let $P_{\underline{j+1}}$=$P_{(j+1,d-j)}\subset G$ is the (lower) standard-parabolic  subgroup attached to the decomposition $(j+1,d-j)$
of $d+1.$ It is clearly the stabilizer of $\P^j_K$ under the above action.  Both the Zariski cohomology  $H^\ast(\P^d_K\setminus \P^j_k,\cE)$ and the algebraic local cohomology 
$H^\ast_{\P^j_K}(\P^d_K,\cE)$ are thus equipped with an action of the semi-direct product
$P_{(j+1,d-j)} \ltimes U(\frg)$. Here the semi-direct product
is as usual induced by the adjoint action of $P_{(j+1,d-j)}$ on $\frg.$ 
Further the natural long exact sequence \vspace{0.3cm}
\begin{equation}\label{long_exact}
\cdots \rightarrow H^{i-1}(\P^d_K\setminus \P^j_K,\cE) \rightarrow H^i_{\P^j_K}(\P^d_K,\cE) \rightarrow H^i(\P^d_K,\cE) \rightarrow H^i(\P^d_K\setminus \P^j_K,\cE) \rightarrow \cdots
\end{equation}
is equivariant with respect to this action. By general arguments in local cohomology theory \cite{Ha2}, one deduces that
\begin{equation}\label{cases}
 H^{i}_{\P^j_K}(\P^d_K,\cE)=\left\{\begin{array}{ccc}
0 \, & ; &  i< d-j \\ \\ H^{i}(\P^d_K,\cE)\, & ;& i>d-j
 \end{array}
\right. .
\end{equation}
In the case $i=d-j,$ we have thus an exact sequence
\begin{eqnarray*}
\nonumber 0 \rightarrow &  H^{d-j-1}(\P^d_K,\cE) & \rightarrow H^{d-j-1}(\P^d_K\setminus \P^j_K,\cE) \rightarrow  H^{d-j}_{\P^j_K}(\P^d_K,\cE) \\  \rightarrow   & H^{d-j}(\P^d_K,\cE)&  \rightarrow 0.
\end{eqnarray*}
We set
\begin{eqnarray}\label{iso}
 \tilde{H}^{d-j}_{\P^j_K} (\P^d_K ,\cE) & := &{\rm ker}\,\Big(H^{d-j}_{\P^j_K}(\P^d_K,\cE) \rightarrow H^{d-j}(\P^d_K,\cE)\Big) \\
\nonumber & \cong & {\rm coker}\,\Big(H^{d-j-1}(\P^d_K,\cE) \rightarrow H^{d-j-1}(\P^d_K\setminus \P^j_K,\cE)\Big)
\end{eqnarray}
which is consequently a $P_{(j+1,d-j)} \ltimes U(\frg)$-module.

For an arbitrary parabolic subgroup $P\subset G$, let $\cO^\frp$ be the full subcategory of the category
$\cO$ (in the sense of Bernstein, Gelfand, Gelfand \cite{BGG}) consisting of $U(\frg)$-modules of type $\frp=\Lie P$.  
We let $\Oa^\frp$ be the full subcategory of $\cO^\frp$
given by objects $M$ such that all $\frp$-representations appearing in $M$ are induced by  
finite-dimensional algebraic $P$-representations, cf. \cite{OS}.

\begin{Lemma}\label{in category O} The $U(\frg)$-module $\tilde{H}^{d-j}_{\P^j_K} (\P^d_K ,\cE)$ lies in the category $\cO^{\frp_{(j+1,d-j)}}_{\alg}.$
\end{Lemma}

\begin{proof}
 This is an easy consequence of  \cite[Lemma 1.2.1]{O} which states the existence of a finite-dimensional algebraic $P_{(j+1,d-j)}$-module 
   which generates $\tilde{H}^{d-j}_{\P^j_K} (\P^d_K ,\cE)$ as $U(\frg)$-module. 
\end{proof}

The next statement is the main result of \cite{O}. For its formulation we need some more notation.
Denote by $\Rep^{\ell a}_K(G)$ the category of locally analytic  $G$-representations with coefficients in $K.$
For a parabolic subgroup $P\subset G$, let 
$$\Ind^G_P : \Rep_K^{\ell a}(P) \to \Rep_K^{\ell a}(G)$$ be the locally analytic induction functor \cite{F}.
Let ${\rm St}_{d-j}=v^{\GL_{d-j}}_B$ be the smooth Steinberg representation of $\GL_{d-j}(K)$, $j=0,\ldots,d$. We consider ${\rm St}_{d-j}$ as a representation of 
$P_{(j+1,d-j)}$ via the trivial action of the unipotent radical of $P_{(j+1,d-j)}$  and the 
factor $\GL_{j+1}(K) \subset L_{(j+1,d-j)}$, respectively. We equip ${\rm St}_{d-j}$ with the finest locally convex topology so that it becomes 
a locally analytic $P$-representation \cite{ST2}. Thus for any algebraic representation $N$ of $P_{(j+1,d-j)}$, the tensor product
$N\otimes {\rm St}_{d-j}$ is a locally analytic representation. In particular this applies to 
the $G$-representation $H^{d-j}(\P^d_K,\cE)' \otimes v^{G}_{{P_{(j+1,1,\ldots,1)}}}$ which we also denote by
$v^{G}_{{P_{(j+1,1,\ldots,1)}}}(H^{d-j}(\P^d_K,\cE)').$
\medskip

\begin{Theorem}\label{thm}  For $j=0,\ldots,d-1,$ there are
extensions of locally analytic $G$-representa\-tions
$$0 \rightarrow v^{G}_{{P_{(j+1,1,\ldots,1)}}}(H^{d-j}(\P^d_K,\cE)') \rightarrow (\cE(\cX)^{j}/\cE(X)^{j+1})' \rightarrow  
\Ind^G_{P_{\underline{j+1}}}(U_j')^{{\frd_j}=0} \rightarrow 0. $$ 
\end{Theorem}

\begin{proof}
 This is \cite[Theorem 1]{O}.
\end{proof}

\noindent  
Here the $P_{\underline{j+1}}$-representation $U_j'$ is a tensor product 
$N_{j}' \otimes {\rm St}_{d-j}$ of an algebraic $P_{\underline{j+1}}$-representa\-tion $N_{j}'$ and ${\rm St}_{d-j}.$
The symbol $\frd_j$  indicates a system of differential equations depending on $N_j.$
The latter object is characterized  by the property that it generates the kernel of the 
natural  homomorphism $H^{d-j}_{\P^j_K} (\P^d_K, \cE) \rightarrow  H^{d-j}(\P^d_K,\cE)$ 
as a module with respect to $U(\frg)$.

This is exactly the starting point  of the main construction in  \cite{OS}. In fact the locally analytic $G$-representation  
$\Ind^G_{P_{\underline{j+1}}}(U_j')^{{\frd_j}}$ above can be characterized as the image of the object
$\tilde{H}^{d-j}_{\P^j_K} (\P^d_K ,\cE) \times {\rm St}_{d-j}$ under a functor
$$\cF^G_{P_{\underline{j+1}}}: \Oa^{\frp_{\underline{j+1}}}\times \Rep^{\infty}_K(L_{\underline{j+1}}) \longrightarrow  \Rep_K^{\ell a}(G),$$
i.e. 
$$\Ind^G_{P_{\underline{j+1}}}(U_j')^{{\frd_j}}=\cF^G_{P_{\underline{j+1}}}(\tilde{H}^{d-j}_{\P^j_K} (\P^d_K ,\cE),{\rm St}_{d-j}).$$
Here $\Rep^{\infty}_K(L_{\underline{j+1}})$ is the category of smooth  $L_{\underline{j+1}}$-representations with coefficients over $K$.

Let us  briefly recall the definition of this functor for an arbitrary parabolic subgroup $P\subset G$ with Levi decomposition $P=L\cdot U.$
Let $M$ be an object of $\Oa^\frp$. Then  there is a surjective map
\begin{equation*}
\phi:U(\frg) \otimes_{U(\frp)} W \rightarrow M
\end{equation*}
for some finite-dimensional  algebraic $P$-representation  $W\subset M$.
Let $V$ be a additionally a smooth $L$-representation.  We consider $V$ via the trivial action of $U$  as a $P$-representation.
As explained above the tensor product  representation $W' \otimes_K V$  is a locally analytic $P$-representation.
Then
$$\cF^G_{P}(M,V) = \Ind^G_{P}(W' \otimes_K V)^{\frd}$$
denotes the subset of functions $f\in \Ind^G_{P}(W' \otimes _K V)$ which are killed by the ideal $\frd=\ker(\phi).$ In loc.cit. it is shown 
that this subset is a well-defined $G$-stable  closed subspace of
$\Ind^G_{P}(W'\otimes_K V)$ and has therefore a natural structure of a locally analytic $G$-representation. The resulting functor 
is contravariant in the first and covariant in the second variable. 
It is proved in \cite[Prop. 4.10 a)]{OS} that $\cF^G_P$ is exact in both arguments.

Now we come to the functoriality aspect concerning the filtration (\ref{filtration}) mentioned in the introduction.

\begin{Lemma}\label{lemma_functorial}
 Let $f:\cE\to \cE'$ be a homomorphism of homogeneous vector bundles on $\P^d_K$. Then $f$ is compatible with the filtrations,
 i.e., $f$ induces  $G$-equivariant homomorphisms
  $\cE(\cX)^i \to \cE'(\cX)^i$, $i\geq 0.$
\end{Lemma}
 
 \begin{proof}
 The definition of the filtration involves only the geometry of $\cX$ (being the complement of an hyperplane arrangement) and not the
 homogeneous vector bundle itself.
 In fact, the $K$-Fr\'echet space $\cE(\cX)=H^0(\cX,\cE)$ appears in an exact sequence
$$0\rightarrow H^0(\P_K^d,\cE) \rightarrow H^0(\cX,\cE)\rightarrow H^1_{\cY}(\P_K^d,\cE) \rightarrow H^1(\P_K^d,\cE)\rightarrow 0.$$
We consider the $K$-Fr\'echet space  $H^1_{\cY}(\P_K^d,\cE),$ where ${\cY}\subset \P^d_K$ is the "closed" complement of $\cX$ in $\P_K^d.$
The filtration is induced (by taking the preimage) by one on $H^1_{\cY}(\P_K^d,\cE)$ which we briefly review. 
Here all geometric  objects are considered as pseudo-adic spaces in the sense of \cite{Hu}.

Let $\{e_0,\ldots,e_d\}$ be the standard basis of $V=K^{d+1}$ and let $\Delta$ be the standard basis of simple roots with respect to the Borel subgroup of lower triangular matrices. 
For any $\alpha_i\in \Delta,$ put $V_{\alpha_i}=V_i=\bigoplus^i_{j=0} K\cdot e_j$ and set $Y_{\alpha_i}=Y_i=\P(V_i)$.
For any subset $I\subset \Delta$ with  $\Delta\setminus I=\{\alpha_{i_1} < \ldots < \alpha_{i_r}\},$ 
let $Y_I=Y_{\alpha_{i_1}}=\P(V_{i_1}).$ Furthermore, let $P_I$ be the lower parabolic subgroup of $G,$ such that  $I$
coincides with the set of simple roots appearing in the Levi factor of $P_I$. 
Then $gY_I$ is a closed subset of $\cY$ and we denote by $\Phi_{g,I} :gY_I \hookrightarrow \cY$ the corresponding  embedding. Let $\Z$ be the constant sheaf on $\cY$ 
and set  $\Z_{g,I}:=(\Phi_{g,I})_\ast(\Phi_{g,I}^\ast (\Z))$.
Then $$\sideset{}{'}\prod_{g\in G/P_I} \Z_{g,I} \subset \prod_{g\in G/P_I} \Z_{g,I}$$
denotes the subsheaf of locally constant sections with respect to the
topological space $G/P_I.$
In \cite[2.1]{O}  we proved that there is an acyclic resolution
\begin{multline}
\nonumber 0 \rightarrow \Z \rightarrow\!\!\! \bigoplus_{I \subset
\Delta \atop |\Delta\setminus I|=1} \sideset{}{'}\prod_{g\in G/P_I} \Z_{g,I}
\rightarrow \!\!\!\bigoplus_{I \subset \Delta \atop |\Delta\setminus I|=2}
\sideset{}{'}\prod_{g\in G/P_I} \Z_{g,I} \rightarrow \cdots  \\ \label{complex} \\ \nonumber \cdots\rightarrow
\!\!\!\bigoplus_{I \subset \Delta \atop |\Delta\setminus I|=d-1} \sideset{}{'}\prod_{g\in G/P_I} \Z_{g,I}
\rightarrow \sideset{}{'}\prod_{g\in G/P_\emptyset} \Z_{g,\emptyset} \rightarrow 0
\end{multline}
of the constant sheaf $\Z$ on $\cY.$
By applying the functor ${\rm Hom}(i_\ast(\;-\;),\cE)$ to this complex, we get a spectral sequence converging to 
$H_{{\cY}}^1(\P_K^d,\cE).$ Finally the filtration on $H^1_{\cY}(\P_K^d,\cE)$ is just the one induced by this spectral sequence. 
It follows now easily from the construction that $f$ is compatible with the filtrations on $\cE(\cX)$ and $\cE'(\cX).$
 \end{proof}

The de Rham complex (\ref{deRham}) together with   Lemma  \ref{lemma_functorial} induces complexes 
\begin{align*}
 0 \to \cO(\cX)^j/\cO(\cX)^{j+1} \to \Omega^1(\cX)^j/\Omega^1(\cX)^{j+1}  \to \cdots \to \Omega^d(\cX)^j/\Omega^d(\cX)^{j+1}  \to  0 ,
 \end{align*}
 $j=0,\ldots,d-1$,
which form just the $E_0$-term of the spectral sequence attached to the filtered de Rham complex (\ref{filtered_deRham}).
Apart from the terms $v^{G}_{{P_{(j+1,1,\ldots,1)}}}(H^{d-j}(\P^d_K,\Omega^i)')$, $i=0,\ldots,d$,  appearing in Theorem \ref{thm}, this complex 
coincides by what we observed above with the complex
\begin{multline}
 0 \to \cF^G_{P_{\underline{j+1}}}(\tilde{H}^{d-j}_{\P^j_K} (\P^d_K, \Omega^d),{\rm St}_{d-j} ) \to \cdots \\ \cdots \to 
 \cF^G_{P_{\underline{j+1}}}(\tilde{H}^{d-j}_{\P^j_K} (\P^d_K, \Omega^1),{\rm St}_{d-j}) \to
\cF^G_{P_{\underline{j+1}}}(\tilde{H}^{d-j}_{\P^j_K} (\P^d_K, \cO), {\rm St}_{d-j})  \to 0 
\end{multline}
\medskip

\begin{Proposition}
The above complex is acyclic.
\end{Proposition}

\begin{proof}
By the exactness of the functor $\cF^G_P$ in the first entry it suffices to prove that the complex
 $$0 \to \tilde{H}^{d-j}_{\P^j_K} (\P^d_K, \cO) \to \tilde{H}^{d-j}_{\P^j_K} (\P^d_K, \Omega^1) \to 
 \cdots \to \tilde{H}^{d-j}_{\P^j_K} (\P^d_K, \Omega^d) \to 0 $$
of $\frg$-modules is acyclic. Set $V=\P^d_K\setminus \P^j_K= \bigcup_{k=j+1}^d D(T_k).$ Then by identity (\ref{iso}) we have
the description
$$\tilde{H}^{d-j}_{\P^j_K} (\P^d_K, \Omega^i)\cong 
{\rm coker} (H^{d-j-1}(\P^d_K, \Omega^i) \to H^{d-j-1} (V, \Omega^i))$$
for all $i\geq 0.$ 
On the other hand,
we have the following well-known  chain of identities 
\begin{equation}\label{equations}
 K=H^0(\P^d_K,\cO)=H^1(\P^d_K,\Omega^1)=\cdots = H^d(\P^d_K,\Omega^d),
\end{equation}
cf. \cite{Ha1}. All other cohomology groups vanish.
Therefore it is enough to prove that the cohomology in degree $d-j-1$
of the complex
$$0 \to H^{d-j-1} (V, \cO) \to  H^{d-j-1} (V, \Omega^1) \to \cdots \to  H^{d-j-1}(V, \Omega^d) \to 0$$ induces
$H^{d-j-1}(\P^d_K, \Omega^{d-j-1})=K$ and vanishes elsewhere. For this issue, we consider the double complex

\vspace{0.5cm}

\label{bigdoublecomplex}
$$
\begin{array}{ccccc}
\bigoplus\limits_{k=j+1}^d \Omega^d(D(T_k))  &  \rightarrow &  \bigoplus\limits_{j+1\leq k<l \leq d} \Omega^d(D(T_k)\cap D(T_l)) &
\rightarrow \dots \rightarrow & \Omega^d(\bigcap\limits_{k=j+1}^d D(T_k))  \\
\uparrow & & \uparrow &   &\uparrow  \\
\vdots &  & \vdots &  & \vdots  \\
\uparrow & & \uparrow &   &\uparrow   \\
\bigoplus\limits_{k=j+1}^d \Omega^i(D(T_k))  & \rightarrow &  \bigoplus\limits_{j+1\leq k<l \leq d} \Omega^i(D(T_k)\cap D(T_l)) &
\rightarrow \dots \rightarrow & \Omega^i(\bigcap\limits_{k=j+1}^d D(T_k)) \\
\uparrow & & \uparrow &   &\uparrow \\
\vdots  & & \vdots  & &\vdots   \\
\uparrow & & \uparrow &   &\uparrow   \\
\bigoplus\limits_{k=j+1}^d \Omega^1(D(T_k))  & \rightarrow & \bigoplus\limits_{j+1\leq k<l \leq d} \Omega^1(D(T_k)\cap D(T_l))  &
\rightarrow \dots \rightarrow & \Omega^1(\bigcap\limits_{k=j+1}^d D(T_k))  \\
\uparrow & & \uparrow &   &\uparrow   \\
\bigoplus\limits_{k=j+1}^d \cO(D(T_k)) & \rightarrow & \bigoplus\limits_{j+1\leq k<l \leq d} \cO(D(T_k)\cap D(T_l)) &
\rightarrow \dots \rightarrow & \cO(\bigcap\limits_{k=j+1}^d D(T_k))  \\
\end{array}
$$
\vspace{0.3cm}

\noindent whose total complex  gives rise to the de Rham cohomology of $V$, cf. \cite{Gr}. Since 
$H^{k}_{\P^j_K} (\P^d_K, \Omega^i)=0$ for all  $k<d-j$
by identity (\ref{cases}), 
we see  that $H^k(V,\Omega^i)=H^k(\P^d,\Omega^i)$ for all such indices $k.$ 
Evaluating the double complex along the horizontal lines we get thus the  $E_1$-term:

$$
\begin{array}{ccccccc}
0 &  \rightarrow &  0&
\rightarrow \dots \rightarrow & 0 & \rightarrow &   H^{d-j-1} (V, \Omega^d)   \\
\uparrow & & \uparrow &   &\uparrow & &\uparrow  \\
\vdots &  & \vdots &  & \vdots & & \vdots \\
\uparrow & & \uparrow &   &\uparrow &  &\uparrow   \\
0 & \rightarrow &  0 &
\rightarrow \dots \rightarrow & K & \rightarrow  & H^{d-j-1} (V, \Omega^{d-j-2})  \\
\uparrow & & \uparrow &   &\uparrow & & \uparrow  \\
\vdots  & & \vdots  & \rotatebox{75}{$\ddots$} &\vdots&  & \vdots \\
\uparrow & & \uparrow &   &\uparrow & &\uparrow    \\
0  & \rightarrow & K   &
\rightarrow \dots \rightarrow & 0 & \rightarrow  & H^{d-j-1} (V, \Omega^1)  \\
\uparrow & & \uparrow &   &\uparrow & &\uparrow   \\
K & \rightarrow & 0 &
\rightarrow \dots \rightarrow & 0 & \rightarrow & H^{d-j-1} (V, \cO)  \\
\end{array}
$$

%Now $\bigcap D(T_k)\cong \bbA^? \times \bbG_m^?.$ It follows that $H^?(\bbA^? \times \bbG_m^,?)=H^?(\bbG_m^?).$
\medskip
But the de Rham cohomology of  $V$ is easily computed in another way. In fact, using the comparison isomorphism with Betti cohomology \cite{Gr} and the long exact
cohomology sequence for constant coefficients (\ref{long_exact}), we see that  
$H^\ast_{\rm dR}(V)=\bigoplus_{i=0}^{d-j-1} K[-2i].$
 The claim follows now easily.
 \end{proof}

For the proof of Theorem 1, we recall that $E^{p,q}_0=\gr^p(\Omega^{p+q}(\cX)) \Rightarrow H^{p+q}(\Omega^\bullet(\cX))$ is the induced spectral sequence of our 
 filtered de Rham complex.

\begin{Corollary}
The $E_1$-term of the above spectral sequence has the shape
$$E^{p,q}_1=\begin{cases}
            \Hom_K(v^G_{P_{(d+1-p,1,\ldots,1)}},K) & q=0 \\ 0 & q\neq 0
           \end{cases}$$
 for $p\geq 0.$
Hence it degenerates at $E_1$ and we get the formula (\ref{cohomology}).%$$H^\ast_{\rm dR}(\cX)=\bigoplus_{i=0}^d \Hom_K(v^G_{P_{(d+1-i,1,\ldots,1)}},K)[-i].$$ \qed
\end{Corollary} 

This finishes the proof of Theorem 1. \qed

\section{A generalization: The dual BGG complex}

%\vspace{0.7cm}

In this final section we consider a generalization of what we have done before. We replace the de Rham complex 
(\ref{deRham}) by the dual BGG complex attached to an algebraic representation in the 
sense of Faltings \cite{Fa, FC} respectively Schneider \cite{S}. For introducing this complex we have to introduce
some more notation.

Let ${\bf T\subset G}$ be the diagonal torus and let ${\bf B\subset G}$ be the Borel subgroup of lower triangular
matrices. Denote by $\Phi \subset X^\ast({\bf T})$ the corresponding set of roots of ${\bf G}$.
Let ${\bf B^+ \subset G}$ the Borel subgroup of upper triangular matrices and 
and let $\Delta^+ \subset \Phi$ be the set of simple roots with respect to $B^+.$
We consider the set
$$X^+=\{\lambda \in  X^\ast({\bf T}) \mid (\lambda,\alpha^\vee )\geq 0 \; \forall \alpha \in \Delta^+ \}$$
of dominant weights in $X^\ast({\bf T}).$
For $\lambda \in X^+$, we denote by  $V(\lambda)$ the finite-dimensional irreducible algebraic ${\bf G}$-representation 
over $K$ of highest weight $\lambda,$ cf. \cite{Ja}. We consider $V(\lambda)$ as an $G$-representation in the sequel.

Let ${\bf P_{(1,d)}}$ be the stabilizer of the base point $[1:0:\cdots:0] \in
\P^d_K(K)$ and let  ${\bf L}={\bf L_{(1,d)}} \subset {\bf
P_{(1,d)}}$ be the Levi subgroup.
Further let 
$$X^+_L=\{\lambda \in X^\ast({\bf T}) \mid (\lambda,\alpha) \geq 0 \;\forall \alpha \in \Delta_L^+\}$$
be the set of ${\bf L}$-dominant weights where $\Delta^+_L \subset \Delta$ consist of those simple roots which
appear in ${\bf L}$. Every $\lambda\in X^+_L$ gives rise to a  finite-dimensional irreducible  algebraic  ${\bf L}$-representation
$V_L(\lambda).$
We consider it as a ${\bf P}$-module by letting act the unipotent radical trivially on it.
Let
$$\pi:{\bf G} \rightarrow {\bf G/P_{(1,d)}}$$ be the projection map
and identify ${\bf G /P_{(1,d)}}$ with $\P^d_K.$ Let $V$ be a
finite-dimensional algebraic representation of ${\bf P_{(1,d)}}.$
For a Zariski open subset $U\subset \P^d_K$, put
\begin{eqnarray*}\label{homvectorbundle}
\cE_V(U):= &\Big \{ & \mbox{algebraic morphisms } f:\; \pi^{-1}(U)
\rightarrow V \;\mid \; f(gp)\,=\,p^{-1}f(g)\; \mbox{ for all } \\ &
& g \in {\bf G}(\overline{K}), p\in {\bf
P_{(1,d)}}(\overline{K}) \Big\}.
\end{eqnarray*}
Then $\cE_V$ defines a homogeneous vector bundle on $\P^d_K$. We consider it at the same time as such an object over the rigid-analytic space $(\P^d_K)^{\rm rig}$. 
If $\lambda \in X^+_L$ then we set $\cE_\lambda:=\cE_{V_L(\lambda)}.$

Let $W$ be the Weyl group of ${\bf G}$ and consider the dot action $\cdot$  of $W$ on $X^\ast({\bf T})$  given by
$$w\cdot \chi= w(\chi+\rho)-\rho,$$
where $\rho=\frac{1}{2}\sum_{\alpha \in \Phi^+} \alpha.$
Let $W_L\subset W$ be the Weyl group of $L.$ Consider the set $^LW=W_L\backslash W$
of left cosets and the cycles
$$w_i:=(1,2,3,\ldots, i+1) \in S_{d+1}\cong  W,$$
$i=0,\ldots,d$, which are just the representatives of shortest length in their cosets. 
If $\lambda \in X^+$ and $w\in {}^LW$ then
$w\cdot \lambda \in X_L^+$.
The dual BGG-complex of $\lambda \in X^+$ is given by the complex 
\begin{equation}
0 \rightarrow \underline{V}(\lambda) \rightarrow \cE_\lambda \rightarrow \cE_{w_1\cdot \lambda} \rightarrow \cdots \rightarrow
\cE_{w_d\cdot \lambda} \rightarrow 0.  
\end{equation}
Here $\underline{V}(\lambda)$ is the constant sheaf on $\P^d_K$ with values in $V(\lambda).$ By considering sections
in $\cX$ we get a complex 
\begin{equation}
0 \rightarrow V(\lambda) \rightarrow \cE_\lambda(\cX) \rightarrow \cE_{w_1\cdot \lambda}(\cX) \rightarrow \cdots \rightarrow
\cE_{w_d\cdot \lambda}(\cX) \rightarrow 0.  
\end{equation}
It is proved in \cite{S} that the latter one is quasi-isomorphic to the complex $\Omega^\bullet(\cX) \otimes V(\lambda).$
The classical case is \cite{Fa,FC}.
For $\lambda=0$,  we get the usual de Rham complex.

\begin{proof} (of Theorem 1')
 The proof is the same as above. Instead of the series of identities (\ref{equations})
 we use this time the Borel-Weil-Bott theorem, cf. \cite{Ja}. Indeed by considering the spectral sequence 
 $(R^m{\rm ind}^G_P)(R^n{\rm ind}^P_B)(M) \Rightarrow R^n{\rm ind}^G_B(M)$, cf. 
 \cite[Prop. 4.5 c)]{Ja} we deduce that $H^i(\P^d_K,\cE_{w\cdot \lambda})=H^i(G/B,\cL_{w \cdot \lambda})$ since 
 $w\cdot \lambda\in X^+_L$ is $L$-dominant. Here $\cL_{w \cdot \lambda}$ is the line bundle on $G/B$ attached to
 the weight $\lambda$.
 Hence we get 
 $$H^i(\P^d_K,\cE_{w_j\cdot \lambda})=\begin{cases}
            H^0(\P^d_K,\cE_{\lambda}) & i=j \\ 0 & i\neq j .
           \end{cases}$$
 Moreover, the latter object has the description $H^0(\P^d_K,\cE_{\lambda})=V(\lambda).$ 
 As for the interpretation of the de Rham cohomology of $V$ we use the fact \cite{Fa,FC}
 that the complex $\cE_{\bullet\cdot \lambda}(V)$ is quasi-isomorphic to $V(\lambda)\otimes \Omega^\bullet(V)$ instead. 
 The claim follows. \end{proof}

\vspace{1.5cm}
\noindent {\it Acknowledgments:}
I wish to thank Ehud de Shalit and Matthias Strauch for their remarks on this paper.

\end{document}